\documentclass[12pt, a4paper]{amsart}
\usepackage{amsmath}
\usepackage{geometry,amsthm,graphics,tabularx,amssymb,shapepar}
\usepackage{amscd}
\usepackage[all,2cell,dvips]{xy}

\newcommand{\End}{{\mathrm{End}}}

\newcommand{\GL}{{\mathrm{GL}}}

\newcommand{\Hom}{{\mathrm{Hom}}}

\newcommand{\rk}{{\mathrm{k}}}

\newcommand{\Sp}{{\mathrm{Sp}}}

\newcommand{\tr}{{\mathrm{tr}}}

\newcommand{\vsp}{{\vspace{0.2in}}}

\newcommand{\con}{\textit{C}}

\newcommand{\oD}{\textit{D}}

\newcommand{\g}{\mathfrak g}

\renewcommand{\rk}{\mathrm k}

\newcommand{\C}{\mathbb{C}}

\renewcommand{\H}{\mathbb{H}}

\newcommand{\G}{\mathbf{G}}

\newcommand{\E}{\mathbf{E}}
\newcommand{\J}{\mathbf{J}}
\renewcommand{\H}{\mathbf{H}}

\newcommand{\ve}{{\vee}}

\newcommand{\la}{\langle}
\newcommand{\ra}{\rangle}

\newcommand{\be}{\begin {equation}}
\newcommand{\ee}{\end {equation}}
\newcommand{\bee}{\begin {equation*}}
\newcommand{\eee}{\end {equation*}}

\theoremstyle{Theorem}

\theoremstyle{Theorem}
\newtheorem{introconjecture}{Conjecture}

\newtheorem{introtheorem}[introconjecture]{Theorem}

\theoremstyle{Theorem}
\newtheorem{lem}{Lemma}[section]

\newtheorem{leml}[lem]{Lemma}

\theoremstyle{Theorem}
\newtheorem{prp}{Proposition}[section]

\newtheorem{lemp}[prp]{Lemma}

\newtheorem{prpp}[prp]{Proposition}

\theoremstyle{Plain}

\theoremstyle{Definition}

\begin{document}

\title[Multiplicity preservation]{Multiplicity preservation for orthogonal-symplectic  and unitary dual pair correspondences}

\author{Jian-Shu Li}
\author{Binyong Sun}
\author{Ye Tian}

\address{Department of Mathematics, Hong Kong University of Science and Technology, Clear Water Bay, Hong Kong}
\email{matom@ust.hk}
\address{Academy of Mathematics and Systems Science, Chinese Academy of
Sciences, Beijing, 100190, P.R. China} \email{sun@math.ac.cn}

\address{Academy of Mathematics and Systems Science, Chinese Academy of
Sciences, Beijing, 100190, P.R. China}

\email{ytian@math.ac.cn}



\begin{abstract}
Over a non-archimedean local field of characteristic zero, we prove
the multiplicity preservation for orthogonal-symplectic dual pair
correspondences and unitary dual pair correspondences.
\end{abstract}

\maketitle


\section{Introduction}

Fix a non-archimedean local field $\rk$ of characteristic zero, and
a continuous involution $\tau$ on it. Denote by $\rk_0$ the fixed
points of $\tau$. Then either $\rk=\rk_0$ or $\rk$ is a quadratic
extension of $\rk_0$. Let $\epsilon=\pm 1$ and let $E$ be an
$\epsilon$-hermitian space, namely it is a finite dimensional
$\rk$-vector space,  equipped with a non-degenerate $\rk_0$-bilinear
map
\[
  \la\,,\,\ra_E:E\times E\rightarrow \rk
\]
satisfying
\[
     \la u,v\ra_E=\epsilon\la v,u\ra_E^\tau, \quad \la au,v\ra_E=a\la u,
     v\ra_E,\quad a\in A,\, u,v\in E.
\]
Write $\epsilon'=-\epsilon$, and let $(E',\la\,,\,\ra_{E'})$ be an
$\epsilon'$-hermitian space. Then
\[
   \E:=E\otimes_\rk E'
\]
is a $\rk_0$-symplectic space under the form
\[
  \la u\otimes u', v\otimes v'\ra_\E:=\tr_{\rk/\rk_0}(\la u, v\ra_E \, \la
  u',v'\ra_{E'}).
\]

Denote by
\begin{equation}\label{meta}
   1\rightarrow \{\pm 1\}\rightarrow
   \widetilde{\Sp}(\E)\rightarrow
   \Sp(\E)\rightarrow 1
\end{equation}
the metaplectic cover of the symplectic group $\Sp(\E)$. Denote by
\[
 \H:=\E\times \rk_0
\]
the Heisenberg group associated to $\E$, whose multiplication is
given by
\[
   (u,t)(u',t'):=(u+u', t+t'+\la u,u'\ra_\E).
\]
The group $\Sp(\E)$ acts on $\H$ as automorphisms by
\begin{equation}\label{actsp0}
  g.(u,t):=(gu,t).
\end{equation}
It induces an action of $\widetilde{\Sp}(\E)$ on $\H$, and further
defines a semidirect product $\widetilde{\Sp}(\E)\ltimes \H$.

Fix a non-trivial character $\psi$ of $\rk$, and denote by
${\omega}_\psi$ the corresponding smooth oscillator representation
of $\widetilde{\Sp}(\E)\ltimes \H$. Up to isomorphism, this is the
only genuine smooth representation which, as a representation of
$\H$, is irreducible and has central character $\psi$. Recall that
in general, if $H$ is a group together with an embedding of $\{\pm
1\}$ in its center, a representation of $H$ is called genuine if the
element $-1\in H$ acts via the scalar multiplication by $-1$.

Denote by $G$ the group of all $\rk$-linear automorphisms of $E$
which preserve the form $\la\,,\,\ra_E$. It is thus an orthogonal
group, a symplectic group or a unitary group. The group $G$ is
obviously mapped into $\Sp(\E)$. Define the fiber product
\[
  \widetilde{G}:=\widetilde{\Sp}(\E)\times_{\Sp(\E)} G,
\]
which is a double cover of $G$. Similarly, we define $G'$ and
$\widetilde G'$. As usual, the product group $\widetilde{G}\times
\tilde{G'}$ is mapped into $\widetilde{\Sp}(\E)$. \vsp

The goal of this paper is to prove the following theorem, which is
usually called the multiplicity preservation for theta
correspondences, and is also called the Multiplicity One Conjecture
by Rallis in \cite{Ra84}. It is complementary to the famous Local
Howe Duality Conjecture.
\begin{introtheorem}\label{theorem}
For every genuine irreducible admissible smooth representation $\pi$
of $\widetilde{G}$, and $\pi'$ of $\widetilde{G}'$, one has that
\[
   \dim \Hom_{\widetilde{G}\times \widetilde{G}'}(\omega_\psi,
   \pi\otimes \pi')\leq 1.
\]
\end{introtheorem}

When the residue characteristic of $\rk$ is odd, Theorem A is proved
by Waldspurger in \cite{Wa90}. The archimedean analog of Theorem A
is proved by Howe in \cite{Ho89}.

\section{A geometric result}
We continue with the notation of the Introduction. Following
\cite[Proposition 4.I.2]{MVW87}, we extend $G$ to a larger group
$\breve{G}$, which contains $G$ as a subgroup of index two, and
consists pairs $(g,\delta)\in\GL_{\rk_0}(E)\times \{\pm 1\}$ such
that either
\[
  \delta=1 \quad\textrm{and}\quad g\in G,
\]
or
\[
\label{dutilde}
  \left\{
   \begin{array}{ll}
     \delta=-1,&\medskip\\
     g(au)=a^\tau g(u),\quad & a\in \rk,\, u\in E,\quad \textrm{ and}\medskip\\
     \la gu,gv\ra_E=\la v,u\ra_E,\quad & u,v\in E.
   \end{array}
   \right.
\]
Similarly, we define a group $\breve{G'}$ and a group
$\breve{\Sp}(\E)$, which extend $G'$ and $\Sp(\E)$, respectively.

In general, if a group $\breve H$ is equipped with a subgroup $H$ of
index two, we will associate on it the nontrivial quadratic
character which is trivial on $H$. We use $\chi_H$ to indicate this
character.

Denote the fiber product
\[
  \breve{\G}:=\breve{G}\times_{\{\pm 1\}} \breve{G'}=\{(g,g',\delta)\mid (g,\delta)\in \breve{G},\,(g',\delta)\in \breve{G'}\},
\]
which contains
\[
  \G:=G\times G'
\]
as a subgroup of index two. Define a group homomorphism
\begin{equation}\label{xi}
   \begin{array}{rcl}
        \xi:\breve{\G}&\rightarrow &\breve{\Sp}(\E),\smallskip\\
           (g,g',\delta)&\mapsto& (g\otimes g',\delta).\\
   \end{array}
\end{equation}
Let $\breve{\Sp}(\E)$ act on the Heisenberg group $\H$ as group
automorphisms by
\begin{equation}\label{actsp}
  (g,\delta).(u,t):=(gu, \delta t),
\end{equation}
which extends the action (\ref{actsp0}). By using the homomorphism
$\xi$, this induces an action of $\breve{\G}$ on $\H$, and further
defines a semidirect product
\[
  \breve{\J}:=\breve{\G}\ltimes \H,
\]
which contains
\[
  \J:=\G\ltimes \H
\]
as a subgroup of index two.

Let the group
\begin{equation}\label{semid}
  \{\pm 1\}\ltimes (\breve{\G}\times \breve{\G})
\end{equation}
act on $\breve{\J}$ by
\begin{equation}\label{actionall}
   (\delta, \breve{\mathbf g}_1, \breve{\mathbf g}_2). \breve{\mathbf j}:=(\breve{\mathbf g}_1\,
   \breve{\mathbf j}\, \breve{\mathbf g}_2^{-1})^\delta,
\end{equation}
where the semidirect product in (\ref{semid}) is defined by the
action
\[
  -1.(\breve{\mathbf g}_1, \breve{\mathbf g}_2):=(\breve{\mathbf g}_2, \breve{\mathbf g}_1).
\]
The fibre product
\[
   \{\pm 1\}\ltimes_{\{\pm 1\}} (\breve{\G}\times_{\{\pm 1\}}
   \breve{\G})=\{(\delta,\breve{\mathbf g}_1, \breve{\mathbf
   g}_2)\mid \chi_\G(\breve{\mathbf g}_1)=\chi_\G(\breve{\mathbf
   g}_2)=\delta\}
\]
is a subgroup of (\ref{semid}). It contains $\G\times \G$ as a
subgroup of index two, and stabilizes $\J$ under the action
(\ref{actionall}).

We prove the following proposition in the remaining of this section.

\begin{prpp}\label{orbite1}
Every  $\G\times \G$-orbit in $\J$ is stable under the group $\{\pm
1\}\ltimes_{\{\pm 1\}}(\breve{\G}\times_{\{\pm 1\}}\breve{\G})$.
\end{prpp}

\vsp Let $\breve{\G}$ act $\rk_0$-linearly on $\E$  by
\begin{equation}\label{acte}
  (g,g',\delta).u\otimes u':=\delta gu\otimes g'u'.
\end{equation}

\begin{lemp}\label{orbite2}
Every $\G$-orbit in $\E$ is $\breve{\G}$-stable.
\end{lemp}

We first prove

\begin{lemp}\label{orbite3}
Lemma \ref{orbite2} implies Proposition \ref{orbite1}.
\end{lemp}
\begin{proof}
Note that every $\G\times \G$-orbit in $\J$ intersect the subgroup
$\H$, and the subgroup
\[
  \{\pm 1\}\times_{\{\pm 1\}}(\Delta(\breve{\G})) \quad\textrm{ of }\quad \{\pm 1\}\ltimes_{\{\pm 1\}}(\breve{\G}\times_{\{\pm
1\}}\breve{\G}) \] stabilizes $\H$, where ``$\Delta$" stands for the
diagonal group. Therefore in order to prove Proposition
\ref{orbite1}, it suffices to show that every $\Delta(\G)$-orbit in
$\H$ is $\{\pm 1\}\times_{\{\pm 1\}}(\Delta(\breve{\G}))$-stable.
Identify $\{\pm 1\}\times_{\{\pm 1\}}(\Delta(\breve{\G}))$ with
$\breve{\G}$. Then as a $\breve{\G}$-space,
\[
  \H=\E\times \rk,
\]
where $\E$ carries the action (\ref{acte}), and $\rk$ carries the
trivial $\breve{\G}$-action. This finishes the proof.
\end{proof}

Let $\breve{\G}$ act $\rk_0$-linearly on
\[
  \E':=\Hom_\rk(E,E')
\]
by
\[
  ((g,g',\delta).\phi)(u):=\delta \,g'(\phi(g^\tau u)),
\]
where
\[
  g^\tau:=\left\{
         \begin{array}{ll}
           g^{-1}, &\quad
           \textrm{if } \delta=1,\medskip\\
           \epsilon g^{-1},& \quad
           \textrm{if } \delta=-1.\\
         \end{array}
   \right.
\]
Then one checks that the $\rk_0$-linear isomorphism
\[
  \begin{array}{rcl}
    \E&\rightarrow &\E',\\
    u\otimes u'&\mapsto&(v\mapsto \la v,u\ra_E u')
  \end{array}
\]
is $\breve{\G}$-intertwining. Therefore Lemma \ref{orbite2} is
equivalent to the following
\begin{lemp}\label{orbite4}
Every $\G$-orbit in $\E'$ is $\breve{\G}$-stable.
\end{lemp}

Denote by
\[
  \g:=\{x\in \End_\rk(E)\mid \la xu,v\ra_E+\la u,xv\ra_E=0\}
\]
the Lie algebra of $G$, and put
\[
  \tilde{\g}:=\{(x,F)\mid x\in \g, F\,\textrm{ is a $\rk$-subspace of }E,
  \,x|_F=0\}.
\]
Let $\breve G$ act on $\tilde \g$ by
\[
  (g,\delta). (x,F):=(\delta gxg^{-1}, gF).
\]
The action of $\breve{\G}$ on $\E'$ induces an action of
\[
  \breve{G}=\breve{\G}/G'
\]
on the quotient space $G'\backslash\E'$.

\begin{lemp}\label{orbite5}
There is a $\breve G$-intertwining embedding from $G'\backslash\E'$
into $\tilde \g$.
\end{lemp}
\begin{proof}
Recall that the map
\[
  x\mapsto\la x\,\cdot,\,\cdot\ra_E
\]
establishes a $\rk_0$-linear isomorphism form $\g$ onto the space of
$\epsilon'$-hermitian forms on the $\rk$-vector space $E$. Define a
map
\[
   \begin{array}{rcl}
        \Xi: \E'=\Hom_\rk(E,E')&\rightarrow &\tilde \g,\\
             \phi&\mapsto &(x,F),
   \end{array}
\]
where $F$ is the kernel of $\phi$, and $x$ is specified by the
formula
\[
  \la \phi(u),\phi(v)\ra_{E'}=\la xu,v\ra_E,\quad u,v\in E.
\]
Use Witt's Theorem, one finds that two elements of $\E'$ stay in the
same $G'$-orbit precisely when they have the same image under the
map $\Xi$. Therefore $\Xi$ reduces to an embedding
\[
  G'\backslash\E'\hookrightarrow\tilde \g,
\]
which is checked to be $\breve G$-intertwining.
\end{proof}

The following lemma is stated in \cite[Proposition 4.I.2]{MVW87}. We
omit its proof.

\begin{lemp}\label{geometry}
For every $(x,F)\in \tilde \g$, there is an element $(g,-1)\in
\breve{G}$ such that \[
  gxg^{-1}=-x\quad\textrm{ and }\quad gF=F.
\]
\end{lemp}

In other words, every element of $\tilde \g$ is fixed by an element
of $\breve G\setminus G$. Therefore every $G$-orbit in $\tilde \g$
is $\breve G$-stable. Now Lemma \ref{orbite5} implies that every
$G$-orbit in $G'\backslash \E'$ is $\breve G$-stable, or
equivalently, every $\G$-orbit in $\E'$ is $\breve \G$-stable. This
proves Lemma \ref{orbite4}, and the proof of Proposition
\ref{orbite1} is now complete.

\section{Proof of Theorem \ref{theorem}}

We first recall the notions of distributions and generalized
functions on a t.d. group, i.e., a topological group whose
underlying topological space is Hausdorff, secondly countable,
locally compact and totally disconnected. Let $H$ be a t.d. group. A
distribution on $H$ is defined to be a linear functional on
$\con^\infty_0(H)$, the space of compactly supported, locally
constant (complex valued) functions on $H$. Denote by
$\oD^\infty_0(H)$ the space of compactly supported distributions on
$H$ which are locally scalar multiples of a fixed haar measure. A
generalized function on $H$ is defined to be a linear functional on
$\oD^\infty_0(H)$.

Recall the following version of Gelfand-Kazhdan criteria.

\begin{lem}\label{gelfand}
Let $S$ be a closed subgroup of a t.d. group $H$, and let $\sigma$
be a continuous anti-automorphism of $H$. Assume that every
bi-$S$-invariant generalized function on $H$ is $\sigma$-invariant.
Then for every irreducible admissible smooth representations $\pi$
of $H$, one has that
\begin{equation*}
  \dim \Hom_{S}(\pi, \C) \,\cdot\, \dim \Hom_{S}
  (\pi^\ve,\C)\leq 1.
\end{equation*}
\end{lem}
Here and henceforth, we use ``$^\ve$" to indicate the contragredient
of an admissible smooth representation. Lemma \ref{gelfand} is
proved in a more general form in \cite[Theorem 2.2]{SZ} for real
reductive groups. The same proof works here and we omit the details.

\vsp

Now we continue with the notation of the last section.

\begin{leml}\label{invgeneralized}
If a generalized function on $\J$ is $\G\times \G$ invariant, then
it is also invariant under the group  $\{\pm 1\}\ltimes_{\{\pm
1\}}(\breve{\G}\times_{\{\pm 1\}}\breve{\G})$.
\end{leml}
\begin{proof}
Note that the t.d. group $\breve{\J}$ is unimodular. Therefore we
may replace ``generalized function" by ``distribution" in the proof
of the lemma. Then by \cite[Theorem 6.9 and Theorem 6.15 A]{BZ76},
the lemma is implied by Proposition \ref{orbite1}.
\end{proof}

\begin{leml}\label{pmul}
For every irreducible admissible smooth representations $\Pi$ of
$\J$, one has that
\begin{equation*}
  \dim \Hom_{\G}(\Pi, \C) \,\cdot\, \dim \Hom_{\G}
  (\Pi^\ve,\C)\leq 1.
\end{equation*}
\end{leml}
\begin{proof}
The lemma follows from Lemma \ref{gelfand} and Lemma
\ref{invgeneralized} by noting that an element of the form
\[
   (-1,\breve{\mathbf g}, \breve{\mathbf g})\in \{\pm 1\}\ltimes_{\{\pm
1\}}(\breve{\G}\times_{\{\pm 1\}}\breve{\G})
\]
acts as an anti-automorphism on $\J$.
\end{proof}

Let $\omega_\psi$, $\pi$ and $\pi'$ be as in Theorem A. As in the
proof of \cite[Lemma 5.3]{Sun08},
$\omega_{\psi}\otimes\pi^\ve\otimes \pi'^\ve$ is an irreducible
admissible smooth representation of $\J$. Therefore Lemma \ref{pmul}
implies that
\begin{equation}\label{pin}
  \dim \Hom_{\G}(\omega_{\psi}\otimes\pi^\ve\otimes
\pi'^\ve, \C) \,\cdot\, \dim \Hom_{\G}
  (\omega_{\psi}^\ve\otimes\pi\otimes
   \pi',\C)\leq 1.
\end{equation}
By \cite[Theorem 1.4]{Sun09}, the two factors in the left hand side
of (\ref{pin}) are equal. Therefore
\[
   \dim \Hom_{\G}(\omega_{\psi}\otimes\pi^\ve\otimes \pi'^\ve, \C)\leq
   1,
\]
and consequently,
\[
   \dim \Hom_{\G}(\omega_{\psi},\pi\otimes \pi')\leq
   1.
\]
This finishes the proof of Theorem A.


\begin{thebibliography}{99}



\bibitem[BZ76]{BZ76}
I. N. Bernstein and A. V. Zelevinskii, \textit{Representations of
the group GL(n, F) where F is a non-archimedean local field},
Russian Mathematical Surveys (1976) 31, no. 3, 1-68.


\bibitem[Ho89]{Ho89}
R. Howe, Transcending classical invariant theory, J. Amer. Math.
Soc. 2 (1989), no. 3, 535-552.



\bibitem[MVW87]{MVW87}
C. Moeglin, M.-F. Vigneras, and J.-L. Waldspurger, Correspondence de
Howe sur un corp p-adique, Lecture Notes in Math. 1291, Springer,
1987.

\bibitem[Ra84]{Ra84}
S. Rallis, \textit{On the Howe duality conjecture}, Composito Math.
51 (1984), 333-399.

\bibitem[Su08]{Sun08}
B. Sun, \textit{Multiplicity one theorems for symplectic groups},
preprint, 2008.

\bibitem[Su09]{Sun09}
B. Sun, \textit{Contragredients of irreducible representations in
theta correspondences}, preprint, 2009.

\bibitem[SZ08]{SZ}
B. Sun and C.-B. Zhu, \textit{Matrix coefficients for distributional
vectors and an application}, preprint,
http://www.math.nus.edu.sg/$\sim\,$matzhucb/publist.html, 2008.

\bibitem[Wa90]{Wa90}
J.-L. Waldspurger, D\'{e}monstration d'une conjecture de dualit\'{e}
de Howe dans le cas $p$-adique, $p\neq 2$, Festschrift in Honor of
I. I. Piatetski-Shapiro on the Occasion of his Sixtieth Birthday,
Part I (Ramat Aviv, 1989), Israel Math. Conf. Proc., vol. 2,
Weizmann, Jerusalem, 1990, 267-324.


\end{thebibliography}
\end{document}